\documentclass[a4paper,11pt]{amsart}
\usepackage[utf8]{inputenc}
\usepackage[T1]{fontenc}
\usepackage[english]{babel}
\usepackage{mathtools,amssymb,amsthm}
\usepackage{enumitem}
\usepackage{esint}
\usepackage{lmodern}
\usepackage{geometry}
\usepackage{amsbsy}
\usepackage{mathrsfs}

\usepackage{ulem}
\usepackage{color}

\newcommand{\EEE}{\color{black}}

\newcommand \Om{\Omega}

\newcommand \eps{\epsilon}

\newcommand \E{\mathscr{E}}

\newcommand \N{\mathbb{N}}
\newcommand \X{\mathscr{X}}
\newcommand \V{\mathscr{V}}

\newcommand \R{\mathbb{R}}
\newcommand \Sn{{\mathbb{S}^n}}
\newcommand \Rn{\mathbb{R}^n}

\newcommand{\sm}{\setminus}

\newcommand{\sq}{\subseteq}

\newtheorem{theorem}{Theorem}
\newtheorem{corollary}[theorem]{Corollary}
\newtheorem{lemma} [theorem]{Lemma}
\newtheorem{remark} [theorem]{Remark}

\theoremstyle{definition}

\usepackage[colored]{shadethm}

\newcommand \cthe{ \definecolor{shadethmcolor}{rgb}{0.9,0.9,0.8} \definecolor{shaderulecolor}{rgb}{0.9,0.9,0.8} }
\newcommand \clem{\definecolor{shadethmcolor}{rgb}{0.95,0.8,0.95} \definecolor{shaderulecolor}{rgb}{0.95,0.8,0.95}}

\begin{document}

\title{Sharp inequalities for Neumann eigenvalues on the sphere}
\author[D. Bucur]{Dorin Bucur}
\address[Dorin Bucur]{Univ. Savoie Mont Blanc, CNRS, LAMA \\ 73000 Chamb\'ery, France}
\email{dorin.bucur@univ-savoie.fr}

\author[E. Martinet]{Eloi Martinet}
\address[Eloi Martinet]{Univ. Savoie Mont Blanc, CNRS, LAMA \\ 73000 Chamb\'ery, France}
\email{  eloi.martinet@univ-smb.fr}

\author[M. Nahon]{Mickaël Nahon}
\address[Mickaël Nahon]{Univ. Savoie Mont Blanc, CNRS, LAMA \\ 73000 Chamb\'ery, France}
\email{  mickael.nahon@univ-smb.fr}

\begin{abstract}
We prove that the second nontrivial Neumann eigenvalue of the Laplace-Beltrami operator on the unit sphere $\Sn \sq \R^{n+1}$ is maximized by the union of two disjoint, equal, geodesic balls among all subsets of  $\Sn$ of prescribed volume. In fact, the result holds in a stronger version, involving the harmonic mean of the eigenvalues of order $2$ to $n$, and extends to  densities. A (surprising) consequence  occurs on the maximality of a geodesic ball for the first nontrivial eigenvalue under the volume constraint:  the hemisphere inclusion condition of the Ashbaugh-Benguria result can be  relaxed to a weaker one, namely empty intersection with a geodesic ball of the prescribed volume.  Although we do not prove that this last inclusion result is sharp, for a mass less than the half of the sphere, we numerically identify a density with higher first eigenvalue than the corresponding geodesic ball and with support equal to the full sphere ${\mathbb S}^2$.
\end{abstract}
\maketitle
\section{Introduction}
Let $n \ge 2$ and denote by $\Sn$ the unit sphere of dimension $n$ in $\R^{n+1}$. Let $\Om \sq \Sn$ be an open, Lipschitz set of measure $m\in (0,|\Sn|)$. The eigenvalues of the Laplace-Beltrami operator with Neumann boundary conditions on $\Om$  are
$$0= \mu_0(\Om) \le  \mu_1(\Om) \le \dots\le \mu_k(\Om) \to +\infty,$$
multiplicity being counted. If $\Om$ is connected then $\mu_1(\Om)>0$.

For each $k \in \N$ we have
$$\mu_k(\Om) = \min_{S\in{\mathcal S}_{k+1}} \max_{u \in S\sm \{0\}} \frac{\int_\Om |\nabla u|^2 }{\int_\Om u^2 },$$
where ${\mathcal S}_k$ is the family of all subspaces of dimension $k$ in
$H^1(\Om)$, and the integral is defined with the canonical measure on $\Sn$. Then, there exists $u \in H^1(\Om)$, $u \not= 0$ such that
$$
\begin{cases}
-\Delta_\Sn u = \mu_k(\Om) u \mbox { in } \Om,\\
\frac{\partial u}{\partial \nu} = 0  \mbox { on } \partial \Om.
\end{cases}
$$
The aim of this paper is to find sharp upper bounds for the second eigenvalue provided that the volume of the set $\Om$ is prescribed. Shortly, we prove that
\begin{equation}\label{bmn01}
\mu_2(\Om) \le \mu_1(B^{m/2}),
\end{equation}
where   $m=|\Om|$ and  $B^a$ denotes a geodesic ball of volume $a$ in $\Sn$. This means that the set of volume $m$ with maximal second non trivial Neumann eigenvalue is the union of two disjoint geodesic balls of volume $\frac{m}{2}$.

Although this is our main motivation, the result we obtain is more general, in two directions. First, we shall prove the stronger inequality for the harmonic mean of the eigenvalues of order $2$ to $n$
\begin{equation}\label{bmn02}
\sum_{k=2}^n \frac{1}{\mu_k(\Om)} \ge \frac{n-1}{ \mu_1(B^{m/2})}.
\end{equation}
Second, the inequality \eqref{bmn02}  naturally extends to densities with prescribed mass. This is detailed in the last section, but we also refer to \cite{BH19} for an introduction to such problems.

A (somehow surprising) consequence of inequality \eqref{bmn01} is an extension of the result of Ashbaugh-Benguria stating that if  $\Om$ is contained in a hemisphere of $\Sn$ (and so has volume $m$ not larger than $ \frac{|\Sn|}{2}$), the sharp inequality occurs
\begin{equation}\label{eq_ashbaughbenguria}
\mu_1(\Om)\le  \mu_1(B^{m}).
\end{equation}
Indeed, inequality \eqref{bmn01} can be applied as follows: if $\Om_1, \Om_2$ are two disjoint, smooth open subsets of $\Sn$, then
$$\mu_2(\Om_1 \cup \Om_2) \le  \mu_1\left(B^{\frac{1}{2}(|\Om_1|+|\Om_2|)}\right).$$
This means that
$$\min \{\mu_1(\Om_1), \mu_1(\Om_2)\} \le  \mu_1\left(B^{\frac{1}{2}(|\Om_1|+|\Om_2|)}\right).$$
In particular, this implies that the Ashbaugh-Benguria inequality \eqref{eq_ashbaughbenguria} holds when $\Om$ lies in in the complement of a geodesic ball of volume $m$.  The main explanation of the "ease" with which we obtain this result relies on the richness of the class of new test functions we build for testing the second eigenvalue. It is a challenge to understand whether or not this last inclusion condition can be completely dropped; we comment this issue in Section \ref{sec_density}.  

Let us briefly recall the Euclidean context of these questions. In 1954 Szeg\"o \cite{Sz54} proves that in $\R^2$, the inequality
$$\frac{1}{\mu_1(\Om)} + \frac{1}{\mu_2(\Om)}  \ge  \frac{2}{\mu_1(B^{m}_{Euc})}$$
holds in the class of simply connected smooth domains, where $B^m_{Euc}$ is the Euclidian ball of measure $m=|\Om|$. Two years later, Weinberger \cite{We56} proves that $|\Om|^\frac{2}{n} \mu_1(\Om)$ is maximized by the ball without any topological constraint. While the proof of Szeg\"o is based on conformal mappings and  transplantation of the eigenfunctions of the disc, Weinberger builds a suitable set of test functions orthogonal on constants by extending the eigenfunctions of the ball and taking restrictions to the set $\Om$.

In 2009 Girouard, Nadirsahvili and Polterovich prove that in $\R^2$, the second Neumann eigenvalue is maximized by the union of two disjoint balls of mass $\frac{m}{2}$ in the class of simply connected smooth domains. In 2019, the first author and Henrot \cite{BH19} remove the dimensional and the topological constraint and, moreover, prove that the result continues to be valid in the larger class of densities. While the two dimensional result can somehow be related to the proof of Szeg\"o, being based on conformal transplantation, the $n$-dimensional one comes from the construction of Weinberger, which is refined and complemented by a fine topological argument. This last result was later on extended in the hyperbolic space by Freitas and Laugesen \cite{FrLa20}.
To finish the picture in the Euclidean space, it has been conjectured by Ashbaugh and Benguria in \cite{AsBe93} that the inequality
\begin{equation}\label{bmn03}
\sum_{k=1}^{j} \frac{1}{\mu_k(\Om)} \ge \frac{j}{ \mu_1(B^m)}
\end{equation}
holds true in the Euclidean space $\R^n$, for $j=n$.
In  2018, Wang and Xia proved  a weaker version of the inequality, for $j=n-1$. Although this is not the conjectured result, it still gives a stronger inequality than the Weinberger one in dimension $n \ge 3$.

On spheres, the first results are  due to Chavel \cite{Ch80} and to Ashbaugh and Benguria \cite{AB95}. Refining the Weinberger argument, Ashbaugh and Benguria  prove that given $0<m< \frac{|\Sn|}{2}$ among all subsets of a hemisphere of volume $m$, the geodesic ball is the unique maximizer. A similar assertion for domains allowed to go beyond the hemisphere is generally an open question (even if their volumes are not larger than $\frac{|\Sn|}{2}$). We refer to  \cite[Remark (2) after Theorem 5.1]{AB95} and \cite[Remark (2), p. 1085]{AsBe01} for some ideas to extend the inequality to domains on the whole sphere. However, those results are typically working only for subclasses of domains which have, in some sense, their mass balanced on sections of the sphere centered at the "center of mass" point. We comment in Section \ref{sec_density} the question of completely removing any inclusion constraint.

If $\frac{|\Sn|}{2}\le m <|\Sn|$ the inequality for $\mu_1$ is not well understood and it is likely to fail, at least for some values of $m$. The behaviour of the the first  Neumann eigenvalue on geodesic balls of radius larger than $\frac\pi2$ studied in \cite{AB95} support this idea. For instance, the first Neumann eigenvalue of the hemisphere and of the full sphere coincide, so that the monotonicity of the first eigenvalue on geodesic balls with respect to the measure fails to be true, for values of $m$ above the half of the sphere. This is a strong indicator that the inequality may not be true provided the volume is larger than $\frac{|\Sn|}{2}$. We comment this issue in Section \ref{sec_density} and refer to \cite{Ma22} for some numerical computation in support to this assertion.

We point out the paper \cite{FaWe18}, where Fall and Weth analyze critical sets for $\mu_1$ on the sphere.
As well, recently, in \cite{BBC20} the authors prove on spheres an extension of the result of Wang and Xia \cite{WaXi18} for harmonic means, still under the constraint on $\Om$ to be included in a hemisphere.

Here is our main result.

\begin{theorem}\label{bmn05}
Let $\Om\subset \Sn$ be an open, Lipschitz set. Then
\[\sum_{i=2}^{n}\frac{1}{\mu_i\left(\Om\right)}\geq
\sum_{i=2}^{n}\frac{1}{\mu_i\left(B^{|\Om|/2}\sqcup B^{|\Om|/2}\right)}\left(=\frac{n-1}{\mu_1\left(B^{|\Om|/2}\right)}\right).\]
Equality is attained when $\Om$ is  the union of two equal, disjoint  geodesic balls.
\end{theorem}
 Note that the hemisphere inclusion hypothesis is not imposed in this result. We work with domains on the full sphere and with arbitrary measure.
 If $\Om$ has three or more connected components, then $\mu_2(\Om)$ equals to $0$ and the inequality is trivially true. The relevant cases are when $\Om$ is either connected or has two connected components.

\smallskip
As a consequence,  $\mu_2$ is maximal on two disjoint geodesic balls of half measure. Indeed, we get the following.
\begin{theorem}\label{bmn06} Let $\Om\subset \Sn$ be an open Lipschitz set, then
\[\mu_2\left(\Om\right)\leq \mu_1\left(B^{|\Om|/2}\right).\]
\end{theorem}

The proof of Theorems \ref{bmn05} and \ref{bmn06} is going in the same direction as \cite{BH19}. We build a pack of $n$ test functions which are orthogonal to both constants and on the first (unknown) eigenfunction of a set $\Om$. While the idea of the construction is similar to the one of  \cite{BH19} being based on folding (Weinberger based) Ashbaugh-Benguria test functions across a hyperplane containing the center of the sphere, the main difficulty is related to the validity of the topological argument. Roughly speaking, the key difficulty comes from the non uniqueness of the "center of mass" point of an arbitrary domain on a sphere; a domain may have multiple (even infinite) family of such centers for packages of suitable test functions. The uniqueness of such a point was a crucial ingredient in both topological arguments in \cite{BH19} and \cite{FrLa20}.

\smallskip
A consequence of this last result is an extension of the result of Ashbaugh-Benguria on the first eigenvalue.
\begin{corollary}\label{bmn08.1} Let $m\in (0,|\Sn|/2)$ and let $B^m$ be a geodesic ball  of measure $m$  in $\Sn$.
Let $\Om \subset \Sn\sm B^m$ be an open Lipschitz set such that $|\Om|=m$. Then
\[\mu_1\left(\Om\right)\leq \mu_1\left(B^m\right).\]
\end{corollary}
In the result above, the hemisphere inclusion condition of Ashbaugh-Benguria is replaced by a weaker one, namely inclusion in the complement of a geodesic ball of measure $m$. In several other situations the inequality above continues to be true, as, for instance,  for every domain which is disjoint with one if its isometric images. 
 Let us also mention the recent paper \cite{LL22} where the authors use conformal methods to obtain $\mu_1(\Om)\leq \mu_1(B^m)$ for any simply connected set $\Om\subset\mathbb{S}^2$ under the condition $|\Om|=m\leq 0.94 \left|\mathbb{S}^2\right|$. \newline

The issue of dropping completely the inclusion constraint in Corollary \ref{bmn08.1} has not a clear answer and will be discussed in the last section. In fact, we numerically identify a density, with support on the full sphere ${\mathbb S}^2$ and mass less than $2\pi$ which has higher first eigenvalue than the corresponding geodesic ball. The main consequence of this observation is that no mass transplantation argument such as ours can work on the full sphere.\newline

The paper is organized as follows. In the next section  we recall some topological tools and prove the main technical result on which is based the construction of the test functions. The key argument relies on the control, along a continuous deformation, of the number of zeroes of a vector field build upon the test functions.   Section \ref{bmnsec3} is devoted to the proof of Theorem \ref{bmn06} and its corollaries, while   Section \ref{bmnsec4} is devoted to the proof of Theorem \ref{bmn05}. We choose to switch what should be the natural order of the proofs, since the proof of Theorem \ref{bmn06} is a direct  consequence of the topological arguments, while the proof of Theorem \ref{bmn05} requires some extra analysis which may hide the core ideas. Finally, in the last section we comment about possible extensions of those results and bring some numerical support for our assertions.

\section{Topological results}\label{bmnsec2}
\medskip
\noindent{\bf Notations.} We write $\Sn$ the unit sphere of dimension $n$ in $\R^{n+1}$ and set  $z\mapsto z_1,\hdots,z_{n+1}$ its coordinate functions. For any $a\in \Sn$, we let
\[a^-=\{z\in\Sn:z\cdot a<0\},\ a^\bot=\{z\in\Sn:z\cdot a=0\},\ a^+=\{z\in\Sn:z\cdot a>0\}\]
and denote
\[R_a(z)=z-2(z\cdot a)a,\]
\smallskip
\[F_a(z)=\begin{cases}R_a(z) & \text{ when }z\in a^+,\\ z & \text{ when }z\in a^-\sqcup a^\bot.\\ \end{cases}\]
Above $x \cdot y$ denotes the scalar product of $\R^{n+1}$ between the vectors $x$ and $y$ and  $\langle x,y\rangle$ denotes the space spanned by $x,y$.

For any $w,z\in\Sn$, we denote $\pi_z(w)=w-(w\cdot z)z$ the projection of $w$ on $T_z\Sn=z^\bot$, the     tangent hyperplane at the sphere in point $z$.\bigbreak
Let $|\cdot|$ denote the $n$-dimensional Hausdorff measure on $\Sn$. For any $m\in (0,|\Sn|)$, $r\in (0,\pi)$, and $p\in\Sn$, we let $B_{p,r}$ be the geodesic ball in $\Sn$ of center $p$ and radius $r$, $B_p^m$ the ball of center $p$ and measure $m$. When $p$ is not mentioned we may  implicitly assume that $p=e_{n+1}$ (as it will not matter).\bigbreak

Let $g\in\mathcal{C}^{0}([-1,1],\R_+)$ be a continuous even function such that $g>0$ on $(-1,1)$. We denote $G(t)=\int_0^t g(s) ds$ the primitive of $g$ which vanishes at $t=0$.\bigbreak

For any $\rho\in L^1(\Sn,\R)$, and $(a,z)\in\Sn\times\Sn$ we denote
\[E_\rho(z):=\int_{\Sn}\rho(v)G(z\cdot v)dv\]
and
\[\E_\rho(a,z):=\int_{\Sn}\rho(v)G(z\cdot F_a v)dv.\]
Notice in particular that $\nabla E_\rho(z)=\nabla_z\E_\rho(-z,z)$, where ``$\nabla_z$'' is the partial gradient in its second variable.

\medskip
\noindent{\bf Counting zeroes modulo 4.}
 Let $M$ be a compact manifold with boundary $\partial M$. Suppose that we have some smooth involution without fixed point $S:M\to M$, and a $(1,1)$ tensor $Q$ acting as a invertible linear transformation on each tangent space, meaning that for every $x\in M$ we have $Q(x)\in GL(T_xM)$.\bigbreak
We denote by $\X(M)$ be the set of continuous (tangent) vector fields $V$ on $M$ (with the natural $\mathcal{C}^0$ topology) such that
\[S_* V=QV,\]
meaning that for every $x\in M$, we have $dS(S^{-1}(x))V(S^{-1}(x))=Q(x)V(x)$. We also let
\[\X^*(M)=\{V\in\X(M):V_{|\partial M}\text{ does not vanish}\}.\]
When $V\in\X^*(M)$ is smooth ($\mathcal{C}^\infty$), we say it is nondegenerate if for every $x\in M$ such that $V(x)=0$, and any local chart $\varphi:\mathcal{B}\to M$ (where $\mathcal{B}$ is a ball in Euclidian space) with $\varphi(0)=x$, then $D(\varphi^*V)(0)$ is invertible (where $\varphi^*V=\left(\varphi^{-1}\right)_*V$).
\begin{lemma}\label{TopLemma}
There is a unique continuous function $I:\X^*(M)\to \mathbb{Z}/4\mathbb{Z}$ such that for every smooth nondegenerate vector field $V\in\X^*(M)$,
\[I(V)=\mathrm{Card}\left(\{x\in M:V(x)=0\}\right)\text{ mod }(4).\]
\end{lemma}
In particular $I$ is invariant by homotopy in $\X^*(M)$. This result is similar to \cite[§4]{M97}, \cite[Ch. 7.4]{S11}; we give a proof in the appendix.\newline

The main  technical result on which is based the construction of test functions is the following.
\begin{theorem}\label{bmn08}
Let $\rho,\sigma\in L^1(\Sn,\R)$ with $\rho\geq 0$. Then there exists $(a,z)\in\Sn\times \Sn$ such that $z$ is a critical point of both $\E_\rho(a,\cdot)$  {and} $\E_\sigma(a,\cdot)$.
\end{theorem}
Criticality above reads
$$\int_{\Sn}\rho(v)g(z\cdot F_a (v))\pi_zF_a (v)dv= \int_{\Sn}\sigma(v)g(z\cdot F_a (v))\pi_z F_a (v)dv=0.$$
The remaining part of the section is devoted to the proof of this theorem.
\begin{proof}
 This is direct for any $a$ when $\rho=0$, so we suppose without loss of generality that $\int_{\Sn}\rho>0$. We begin with a computation which  will be useful in several steps of the proof.
\clem\begin{lemma}
Let $p \in \Sn$. For any $r\in (0,\pi/2)$, let us denote $G_r: [-1,1] \to \R$ given by
\[G_r(p\cdot z)= E_{1_{B_{p,r}}}(z).\]
Then, $G_r\in \mathcal{C}^0([-1,1],\R)\cap\mathcal{C}^1((-1,1),\R)$ and satisfies $G_r'>0$ in $(-1,1)$. Moreover,
$$G_r'(p\cdot z)\pi_z(p)=\nabla E_{1_{B_{p,r}}}(z).$$

\end{lemma}

\begin{proof}
By explicit computation,
\[E_{1_{B_{p,r}}}(z)=\int_{B_{p,r}}G(z\cdot v)dv.\]
Since $G$ is $\mathcal{C}^1$, we obtain directly that $E_{1_{B_{p,r}}}$ is $\mathcal{C}^1$ and only depends on the scalar product $z\cdot p$.
Consequently, it may be written as a function $G_r(p\cdot z)$ with $G_r\in\mathcal{C}^1((-1,1),\R)$.

Let us now check that $G_r'>0$ on $(-1,1)$. For any $z\neq \pm p$,
\[G_r'(z\cdot p)\pi_z (p)=\nabla E_{1_{B_{p,r}}}(z)=\int_{B_{p,r}}g(z\cdot v)\pi_z(v)dv.\]
We let $u=-\frac{\pi_z(p)}{|\pi_z(p)|}$ be the unique unit vector in $T_z \Sn\cap  \langle p,z\rangle$ such that $u\cdot p<0$. Then
\begin{align*}
G_r'(z\cdot p)u\cdot p&=\int_{B_{p,r}}g(z\cdot v)(u\cdot v)dv\\
&=\int_{B_{p,r}\setminus R_u(B_{p,r})}g(z\cdot v)(u\cdot v)dv <0.
\end{align*}
The last inequality holds because $w\in B_{p,r}\setminus R_u(B_{p,r})$ implies $u\cdot w<0$.

 \end{proof}

\noindent {\bf Homotopy and zeroes of the gradient fields.}
Let $r\in (0,\pi/4)$ (its precise value is not important), and some $p,q\in\Sn$ that will be fixed later. Consider the following homotopy for $t \in [0,1]$:
\[H_t(a,z):=\left((1-t)\E_\rho(a,z)+t\E_{1_{B_{p,r}}}(a,z),(1-t)\E_\sigma(a,z)+t\E_{1_{B_{q,r}}}(a,z)\right).\]

We claim that $\nabla_z H_t(a,z)$ has no zeroes when $a\cdot z=0$. Indeed looking at the first component we always have

\begin{align*}
a\cdot\nabla_z \E_\rho(a,z)&=\int_{\Sn}\rho(v)g(z\cdot F_a v)(a\cdot\pi_zF_av) dv=\int_{\Sn}\rho(v)g(z\cdot v)(a\cdot F_av) dv<0.\\
a\cdot\nabla_z \E_{1_{B_{p,r}}}(a,z)&=\int_{B_{p,r}}g(z\cdot F_a v)(a\cdot\pi_zF_av) dv=\int_{B_{p,r}}g(z\cdot v)(a\cdot F_av) dv<0.
\end{align*}
so, by uniform continuity, the first component of $\nabla_z H_t(a,z)$ doesn't vanish for any $t\in [0,1]$ and $(a,z)\in\Sn\times \Sn$ for which $|a\cdot z|$ is small enough.\bigbreak

We claim that for well-chosen points $p,q$,
$$z\in a^-,\ \nabla_z H_1(a,z)=0 \mbox{  implies } \{z,R_a (z)\}=\{p,q\}.$$
The points $p,q$ we will be chosen to be antipodal, but it may be checked that any $p,q$ such that $B_{p,r}\cap B_{q,r}=\emptyset$ would work as well.

Indeed,
\begin{itemize}[label=\textbullet]
\item If $\nabla_z\E_{1_{B_{p,r}}}(a,z)=0$, then $p\in \langle a,z\rangle$. This is because for any $b\in \langle z,a\rangle^{\bot}$, we have
\[b\cdot \nabla_z\E_{1_{B_{p,r}}}(a,z)=\int_{B_{p,r}}g(z\cdot F_a (v))(b\cdot v)dv=\int_{B_{p,r}\setminus B_{R_b p,r}}g(z\cdot F_a (v))(b\cdot v)dv,\]
which is positive (resp. negative) as soon as $b\cdot p>0$ (resp. $b\cdot p<0$).
\item If $R_a(z)=-z$ or $B_{p,r}\subset a^-\sqcup a^+$, then $p\in \{z,R_a(z)\}$. Indeed in the first case it means $a=-z$, so $z$ can be identified with the critical points of $E_{1_{B_{p,r}}}$, which are $\{\pm z\}$. In the second case, if $\overline{B_{p,r}}$ is fully in $a^-$ (resp $a^+$), then $z$ (resp $R_a z$) is also a critical point of $E_{B_{p,r}}$, meaning
\[p=\begin{cases}z & \text{ if }p\in a^-,\\ R_a(z) & \text{ if }p\in a^+\end{cases}.\]
\item If $\nabla_z\E_{1_{B_{p,r}}}(a,z)=0$ and $p\notin \{z,R_a (z)\}$, then
\begin{equation}\label{eq_milieu}
c:=\frac{z+R_a(z)}{|z+R_a(z)|}\in \overline{B_{p,r}}.
\end{equation} Indeed, according to the previous point $R_a(z)\neq -z$ (meaning $c$ is well-defined) and $B_{p,r}$ meets $a^\bot$. Since $\overline{B}(p,r)$ and $a^\bot$ intersect, $r<\frac{\pi}{4}$ and $p\in \langle a,z\rangle$, then necessarily $\overline{B_{p,r}}$ contains (exactly) one of $\{-c,c\}$. However, if it meets $-c$ then since $r<\frac{\pi}{4}$, $B_{p,r}\subset c^-\subset\{v:a\cdot \pi_z F_a (v)<0\}$ so
\[a\cdot\nabla_z \E_{1_{B_{p,r}}}(a,z)=\int_{B_{p,r}}g(z\cdot F_a (v))(a\cdot \pi_z F_a (v))dv<0\]
so necessarily $c\in\overline{B_{p,r}}$.
\end{itemize}
Since we are free to choose $p$ and $q$, we may take $p=e_{n+1}$, $q=-e_{n+1}$. Assume $\nabla_z H_1(a,z)=0$ and suppose that $\{p,q\}\neq \{z,R_az\}$. Then using the second point, necessarily $z\neq R_a z$ and either $B_{p,r}$ or $B_{q,r}$ meets $c$ (which is as defined in equation \eqref{eq_milieu}) ; without loss of generality assume $c\in \overline{B_{p,r}}$, then since $r<\frac{\pi}{4}$ we have $B_{q,r}\subset c^-\cap (a^+\sqcup a^-)$, which is in contradiction with the fact that $q\in \{z,R_a z\}$. This prove the claim.\bigbreak

\noindent {\bf Counting zeroes modulo $4$.} Let us now make a change of parametrization; for any $(z,w)\in \Sn\times \Sn$ that are not identical, we define
\[a= a(z,w):=\frac{w-z}{|w-z|},\]
such that $z\in a^-$, $w=R_a z$, $R_a$ acting as an isometry between $T_z \Sn$ and $T_w \Sn$. We let $D$ be a (small) tubular neighbourhood of $\{(z,w)\in(\Sn)^2 :z=w\}$, $M=(\Sn)^2\setminus D$, and we define, for any densities $(\rho,\sigma)$, the vector field
\[\V_{\rho,\sigma}(z,w)=\left(\nabla_z\E_\rho(a(z,w),z),R_a\nabla_z\E_\sigma(a(z,w),z)\right).\]
$\V_{\rho,\sigma}$ and $\V_{1_{B_{p,r}},1_{B_{q,r}}}$ are two tangent vector fields of $M$ that are homotopic with no zeroes crossing $\partial M$ when $D$ is chosen small enough.\newline

Define $S:(z,w)\in M\mapsto (w,z)\in M$; it is a smooth involution with no fixed point. For any $(z,w)\in M$, let $Q(z,w)$ be the linear endomorphism of $T_{z,w}M=T_z\Sn\times T_w\Sn$ defined by
\[Q(z,w).(h,k)=(R_a k,R_a h).\]
We claim that for any $(z,w)\in M$:
\begin{equation}\label{eq_symmetry}
dS(w,z)\V_{\rho,\sigma}(w,z)=Q(z,w)\V_{\rho,\sigma}(z,w),
\end{equation}
or in a more synthetic way that $S_*\V_{\rho,\sigma}=Q\V_{\rho,\sigma}$ as in the hypothesis of Lemma \ref{TopLemma}. Indeed, write $a:=a(z,w)$, notice that $a(w,z)=-a$, that $\pi_z R_a =R_a \pi_w$ and $F_{-a}=R_aF_a$. Then
\begin{align*}
\V_{\rho,\sigma}(w,z)&=\left(\int_{\Sn}\rho(v)g(w\cdot F_{-a}v)\pi_wF_{-a}v dv,R_a\int_{\Sn}\sigma(v)g(w\cdot F_{-a}v)\pi_wF_{-a}v dv\right)\\
&=\left(R_a\int_{\Sn}\rho(v)g(z\cdot F_{a}v)\pi_zF_{a}v dv,\int_{\Sn}\sigma(v)g(z\cdot F_{a}v)\pi_z F_{a}v dv\right),\\
\end{align*}
and $dS(w,z).(h,k)=(k,h)$, so
\begin{align*}
dS(w,z).\V_{\rho,\sigma}(w,z)&=\left(\int_{\Sn}\sigma(v)g(z\cdot F_{a}v)\pi_z F_{a}v dv,R_a\int_{\Sn}\rho(v)g(z\cdot F_{a}v)\pi_zF_{a}v dv\right)\\
&=Q(z,w).\V_{\rho,\sigma}(z,w).
\end{align*}
We thus define the zero counting modulo $4$ in $M$ as in Lemma \ref{TopLemma} and we claim that
\[I(\V_{1_{B_{p,r}},1_{B_{q,r}}})=2\text{ mod }(4).\]
According to the previous discussion, the zeroes of $\V_{1_{B_{p,r}},1_{B_{q,r}}}$ are exactly $\{(p,q),(q,p)\}$. Let $(z,w)$ be sufficiently close to $(p,q)$ such that $\overline{B_{p,r}}\subset a(z,w)^-$, $\overline{B_{q,r}}\subset a(z,w)^+$. Notice that in this case
\begin{align*}
\nabla_z \E_{1_{B_{p,r}}}(a,z)&=\int_{B_{p,r}}g(z\cdot v)\pi_z vdv=G_r'(z\cdot p)\pi_z (p)\\
\nabla_z \E_{1_{B_{q,r}}}(a,z)&=\int_{B_{q,r}}g(z\cdot R_a(v))\pi_z (R_a(v))dv\\
&=R_a\Big (\int_{B_{q,r}}g(w\cdot v)\pi_w (v)dv\Big)\\
&=G_r'(w\cdot q)R_a(\pi_w (q))
\end{align*}
because $\pi_z R_a=R_a \pi_w$. Thus $\V_{1_{B_{p,r}},1_{B_{q,r}}}(z,w)$ admits the simpler expression (when $(z,w)$ is in a neighbourhood of $(p,q)$)
\begin{equation}\label{eq_J}
\V_{1_{B_{p,r}},1_{B_{q,r}}}(z,w)=\left(G_r'(z\cdot p)\pi_z (p),G_r'(w\cdot q)\pi_w(q)\right).
\end{equation}
Up to a choice of orientable chart, this is locally homotopic to $-\mathrm{Id}_{|\mathcal{B}}$ (where $\mathcal{B}$ is the unit Euclidian ball of $\R^{2n}$) which has a unique nondegenerate zero. When $(z,w)$ is near $(q,p)$ the study is the same by the symmetry property \eqref{eq_symmetry}, so $(q,p)$ counts as one nondegenerate zero. We thus get that $I(\V_{1_{B_{p,r}},1_{B_{q,r}}})=2\text{ mod }(4)$ as claimed. Since $\V_{\rho,\sigma}$ and $\V_{1_{B_{p,r}},1_{B_{q,r}}}$ are homotopic in $\X^*(M)$, by invariance of $I$ through homotopy:
\[I\left(\V_{\rho,\sigma}\right)=I\left(\V_{1_{B_{p,r}},1_{B_{q,r}}}\right)=2\text{ mod }(4).\]

\noindent{\bf Conclusion.} Since  $I\left(\V_{\rho,\sigma}\right)\neq 0\text{ mod }(4)$ we conclude that $\V_{\rho,\sigma}$ has a zero $(z,w)$ somewhere in $M$. Since $z\neq w$, then  $(a,z):=\left(\frac{w-z}{|w-z|},z\right)$ satisfies the conclusion of Theorem \ref{bmn08}.\end{proof}

\section{Proof of Theorem \ref{bmn06}}\label{bmnsec3}
Although Theorem \ref{bmn06} is {\it de facto} a consequence of Theorem \ref{bmn05}, for expository reasons, we start with a short, independent  proof of Theorem \ref{bmn06}. The reason is that Theorem \ref{bmn05} requires some extra arguments related to the presence of  higher order eigenvalues and so it can hide the core ideas of the proof.

\begin{proof}(of Theorem \ref{bmn06}) 
Let $r$ be the radius of $B^{|\Om|/2}$. The first non-trivial eigenvalue of $B_r$ (supposed to be centered in $e_{n+1}$) has multiplicity $n$ and its eigenfunctions are of the form
\[v\mapsto J(\theta)\frac{v_i}{\sin(\theta)},\ i=1,\dots, n\]
where $\theta=\arccos(v_{n+1})$ is the angle between $v$ and $e_{n+1}$. The function $J:[0,r]\to \R_+$
is a non-trivial solution of the differential equation
\[\frac{1}{\sin(\theta)^{n-1}}\frac{d}{d\theta}\left[\sin(\theta)^{n-1}\frac{d}{d\theta}J(\theta)\right]+\left(\mu_1(B_r)-\frac{n-1}{\sin(\theta)^2}\right)J(\theta)=0.\]

In \cite{AB95} (see also \cite{BBC20}), the authors study this function and prove that $J(0)=0$, $J'(r)=0$, $J'>0$ on $(0,r)$, meaning $J$ is positive and increasing on $(0,r]$.
 Following \cite{AB95}, we extend $J$ on $[0,\pi]$: we let it be constant (equal to $J(r)$ by continuity) on $[r,\pi/2]$, and symmetric along the reflexion $\theta\mapsto  \pi -\theta$. Moreover, from \cite{AB95} (see also \cite{BBC20}), the functions
\[\theta\in \left(0,\frac{\pi}{2}\right)\mapsto J(\theta)^2,\ \theta\in\left(0,\frac{\pi}{2}\right)\mapsto  J'(\theta)^2+\frac{n-1}{\sin(\theta)^2}J(\theta)^2 \]
are respectively nondecreasing and decreasing.\bigbreak

We then define for all $ t$ in $  (-1,1)$
\[g(t):=\frac{J(\arccos(t))}{\sqrt{1-t^2}}\]
and extend it by continuity in $1$ and $-1$ with value $0$.
We set  $G(t)=\int_{0}^t g$ as previously.
 \bigbreak

Consider $\rho=1_\Om$, $\sigma=u_1 1_\Om$ where $u_1$ is a non-constant eigenfunction associated to $\mu_1(\Om)$. When $\Om$ is disconnected, such a function is  a constant different from $0$ on some of the connected components and equals to $0$ elsewhere.

In order to prove the inequality $\mu_2(\Om) \le \mu_2\left(B^{m/2} \sqcup B^{m/2}\right)$ the mass transplantation method used in \cite{BH19} works, in relation with the properties of $J$ recalled above. Indeed, following Theorem \ref{bmn08} there exists $(a,z) \in \Sn \times \Sn$ such that
$$\int_{\Om}g(z\cdot F_a (v))\pi_zF_a (v)dv= \int_{\Om}u_1(v)g(z\cdot F_a (v))\pi_zF_a (v)dv=0.$$

Up to a rotation we may 	assume $z=e_{n+1}$ with $a\in e_{n+1}^-$. This means that for any $i=1,\hdots,n$, the function
\[v\mapsto g(z\cdot F_a v)(F_a v)_i\]
is orthogonal to $1$ and $u_1$ in $L^2(\Om)$. In particular, for each $i=1,\hdots,n$ we have
\begin{equation}\label{eq_testfunc}
\mu_1(\Om)\int_{\Om}\left|g(z\cdot F_a v)(F_a v)_i\right|^2dv\leq  \int_{\Om}\left|\nabla(g(z\cdot F_a v)(F_a v)_i)\right|^2dv
\end{equation}
We now define $H_{a^-}=\left(B(e_{n+1},r)\cup B(-e_{n+1},r)\right)\cap a^-$, and we define the first angular coordinate $\theta^-:v\mapsto \arccos(v_{n+1})$. Similary, we let $H_{a^+}=R_a(H_{a^-})(\subset a^+)$ and $ {\theta^+}(v)=\theta^-(R_av)$. Then summing the previous inequality in $i$ we get \begin{align*}
\mu_2(\Om)&\left[\int_{H_{a^-}\cap\Om}J(\theta^-)^2dv+\int_{H_{a^+}\cap\Om}J( {\theta^+})^2dv+|\Om\setminus (H_{a^-}\cup H_{a^+})|J(r)^2\right]\\
\leq &\left[\int_{H_{a^-}\cap\Om}b(\theta^-)dv+\int_{H_{a^+}\cap\Om}b( {\theta^+})dv+\int_{\Om\setminus (H_{a^-}\cup H_{a^+})}b\left(\theta^-(F_a v)\right)dv\right]
\end{align*}
Above, we have denoted $b(\theta):=J'(\theta)^2+\frac{n-1}{\sin(\theta)^2}J(\theta)^2$.

Using the properties of $J(\theta)$, $b(\theta)$ and the fact that $|\Om\setminus (H_{a^-}\cup H_{a^+})|=|H_{a^-}\setminus\Om|+|H_{a^+}\setminus \Om|$ this reduces to
\begin{align*}
\mu_2(\Om)&\leq\frac{\int_{H_{a^-}\cap\Om}b(\theta^-)dv+\int_{H_{a^+}\cap\Om}b({\theta^+})dv+|\Om\setminus (H_{a^-}\cup H_{a^+})|b(r)}{\int_{H_{a^-}\cap\Om}J(\theta^-)^2dv+\int_{H_{a^+}\cap\Om}J({\theta^+})^2dv+|\Om\setminus (H_{a^-}\cup H_{a^+})|J(r)^2}\\
&\leq\frac{\int_{H_{a^-}}b(\theta^-)dv+\int_{H_{a^+}}b({\theta^+})dv}{\int_{H_{a^-}}J(\theta^-)^2dv+\int_{H_{a^+}}J( {\theta^+})^2dv}\\
&=\mu_1\left(B_r\right),
\end{align*}
which is the result.

\end{proof}

\cthe\begin{corollary}
Let $\Om\subset\Sn$ be a Lipschitz domain and $i(\Om) \sq \Sn$ an isometric image of $\Om$.  Assume that $\Om\cap i(\Om)= \emptyset$. Then
\[\mu_1(\Om)\leq \mu_1(B^{|\Om|}).\]
\end{corollary}
The proof is an immediate consequence of Theorem \ref{bmn06} applied to to $\Om\sqcup i(\Om)$.
This result was previously known only for $i(\Om)=-\Om$.
\cthe\begin{corollary}\label{bmn15}
Let $\Om\subset\Sn$ such that $|\Om|\le \frac{|\Sn|}{2} $ and $\Om\cap B^{|\Om|}=\emptyset$. Then
\[\mu_1(\Om)\leq \mu_1(B^{|\Om|}).\]
\end{corollary}
This result extends the hemisphere inclusion condition of Ashbaugh and Benguria expressed, with our notation, as  $\Om\cap B^{|\Sn|/2} =\emptyset$.
\begin{proof}

We remind that the function
\[m\in \left(0,|\Sn|\right)\mapsto \mu_1(B^m)\]
is continuous and decreasing on $(0,|\Sn|/2)$ (and actually on a slightly larger interval), as was proved  in \cite{AB95}.

Let $\Om \sq \Sn$ satisfying the hypotheses of Corollary \ref{bmn15}. Consider $\eps\in (0,|\Om|)$, then $\Om\cap B^{|\Om|-\eps}=\emptyset$ and, according to Theorem \ref{bmn06},
\[\min\left\{\mu_1(\Om), \mu_1\left(B^{|\Om|-\eps}\right)\right\}=\mu_2\left(\Om\sqcup B^{|\Om|-\eps}\right)\leq \mu_1\left(B^{|\Om|-\frac{1}{2}\eps}\right)\]
Since $\mu_1\left(B^{|\Om|-\eps}\right)> \mu_1\left(B^{|\Om|-\frac{1}{2}\eps}\right)$, we get $\mu_1(\Om)\leq \mu_1\left(B^{|\Om|-\frac{1}{2}\eps}\right)$, which implies the result as $\eps\to 0$.
\end{proof}

\section{Proof of Theorem \ref{bmn05}}\label{bmnsec4}
\begin{proof}
 In \cite{BBC20}, it has been proved that
\[\sum_{i=1}^{n-1}\frac{1}{\mu_i\left(\Om\right)}\geq
\sum_{i=1}^{n-1}\frac{1}{\mu_i\left(B^{|\Om|}\right)}\left(=\frac{n-1}{\mu_1\left(B^{|\Om|}\right)}\right).\]
We mostly rely on those computations, to which we refer the reader.
With the notations of the proof of Theorem \ref{bmn06}, the main idea is to choose  a suitable basis for $e_{n+1}^\perp$ so that
extra orthogonality conditions on the eigenfunctions occur. This is also the idea behind \cite{WaXi18} and \cite{BBC20}, however, some extra care is needed because of the structure of  the test functions found in Section \ref{bmnsec2}.

We fix again $\rho=1_\Om$, $\sigma=u_1 1_{\Om}$ and
suppose, up to a rotation, that  $a\in e_{n+1}^-$ is such that $e_{n+1}$ is a critical point of both $\E_\rho(a,\cdot)$ and $\E_{\sigma}(a,\cdot)$. This means that for any $\xi\in\Rn(\hookrightarrow \Rn\times\{0\}=e_{n+1}^\perp)$, the function
\[\varphi_\xi(w):=g((F_aw)_{n+1})(\xi\cdot F_aw)\]
is orthogonal to $1,u_1$ in $L^2(\Om)$.
\clem\begin{lemma}
There is a choice of an orthonormal basis $\xi_1,\hdots,\xi_n$ of $\R^n$ such that for any $i=1,\hdots,n$,
\[\varphi_{\xi_i}\in \langle 1,u_1,\hdots,u_{i}\rangle^\bot\]
\end{lemma}
\begin{proof}
We start by choosing $\xi_n$. Consider the function
\[f:\xi\in\mathbb{S}^{n-1}\mapsto \left(\int_{\Om}u_2\varphi_\xi,\hdots,\int_{\Om}u_{n}\varphi_\xi\right)\in \R^{n-1}.\]
Then $f$ verifies $f(-\xi)=-f(\xi)$, so by Borsuk-Ulam theorem \cite{Ma03} we get that $f$ must vanish at some $\xi_n$. We then continue with the restriction
\[\xi\in\mathbb{S}^{n-1}\cap\xi_n^\bot\mapsto \left(\int_{\Om}u_2\varphi_\xi,\hdots,\int_{\Om}u_{n-1}\varphi_\xi\right)\in \R^{n-2}.\]
to choose $\xi_{n-1}$, and so on until $\xi_2$ such that $\varphi_{\xi_2}$ is orthogonal to $1,u_1$ and $u_2$. Finally, $\xi_1$ is then chosen in $\Sn\cap \langle \xi_n,\hdots,\xi_2\rangle^\bot$ (there are two possibilities and we may choose the one that makes an oriented basis, although this is not important for the rest of the proof).
\end{proof}

From the orthogonality properties proved above, for any $i=1, \dots, n$ we get
\[\int_{\Om}| \varphi_{\xi_i}|^2\leq \frac{1}{\mu_{i+1}(\Om)}\int_{\Om}|\nabla \varphi_{\xi_i}|^2.\]
Assuming that $\xi_i =e_i$,

$$\int_{{a^-}} (1_\Om+ 1_{R_a({a^+}\cap\Om)})J^2(\arccos (v_{n+1}))\frac{ v_i^2 }{1-v_{n+1}^2 }dv\le    $$
$$\frac{1}{\mu_{i+1}(\Om)}\int_{{a^-} }  (1_\Om+ 1_{R_a({a^+}\cap\Om)})  \left(\frac{J^2(\arccos (v_{n+1}))(1-v_i^2-v_{n+1}^2)}{(1-v_{n+1}^2)^2} + \frac{(J'(\arccos (v_{n+1})) v_i)^2}{1-v_{n+1}^2}\right) dv.$$
 Since $\theta ^-=\arccos (v_{n+1})$ we simplify
 $$\int_{{a^-}} (1_\Om+ 1_{R_a({a^+}\cap\Om)})J^2(\theta^-)\frac{ v_i^2 }{\sin(\theta^-)^2 }dv\le    $$
$$\frac{1}{\mu_{i+1}(\Om)}\left (\int_{{a^-} }  (1_\Om+ 1_{R_a({a^+}\cap\Om)})   \frac{J^2(\theta^-)}{\sin(\theta^-)^2}\frac{1-v_i^2}{\sin(\theta^-)^2} dv+ \int_{{a^-} }  (1_\Om+ 1_{R_a({a^+}\cap\Om)} )    \frac{(J'(\theta^-) v_i)^2}{\sin(\theta^-)^2}  dv\right).$$
Summing  for $i=1,\hdots,n$, the left hand side becomes
$\int_{{a^-}} (1_\Om+ 1_{R_a({a^+}\cap\Om)})J(\theta^-)^2dv$ and the monotonicity property of $J$ leads to
$$2\int_{H^-} J(\theta^-)^2dv\le  \int_{{a^-}} (1_\Om+ 1_{R_a({a^+}\cap\Om)})J(\theta^-)^2dv.$$
Inside the first term of the right hand side, we use the inequality
\[\sum_{i=1}^{n}\frac{1}{\mu_{i+1}(\Om)}\left(1-\frac{v_i^2}{\sin(\theta^-)^2}\right)\leq \sum_{i=2}^{n}\frac{1}{\mu_i(\Om)},\]
so that
$$\sum_{i=1}^{n} \frac{1}{\mu_{i+1}(\Om)} \int_{{a^-} }  (1_\Om+ 1_{R_a({a^+}\cap\Om)})   \frac{J^2(\theta^-)}{\sin(\theta^-)^2}\frac{1-v_i^2}{\sin(\theta^-)^2} dv\le $$
$$\hskip 6cm \sum_{i=2}^{n}\frac{1}{\mu_i(\Om)}\int_{{a^-} }  (1_\Om+ 1_{R_a({a^+}\cap\Om)})   \frac{J^2(\theta^-)}{\sin(\theta^-)^2}dv.$$
The second term in the right hand side vanishes outside $ ({\Om\cap H^-}) \cup  {R_a({H^+}\cap\Om)}$, hence
$$\sum_{i=1}^{n} \frac{1}{\mu_{i+1}(\Om)}  \int_{{a^-} }  (1_\Om+ 1_{R_a({a^+}\cap\Om)} )    \frac{(J'(\theta^-) v_i)^2}{\sin(\theta^-)^2}  dv\le $$
$$\hskip 6cm
\sum_{i=1}^{n} \frac{1}{\mu_{i+1}(\Om)}  \int_{{H^-} }  (1_\Om+ 1_{R_a({a^+}\cap\Om)} )    \frac{(J'(\theta^-) v_i)^2}{\sin(\theta^-)^2}dv.$$
Relying on the monotonicity property of $b$, we implement now the mass transplantation as follows: if a point belongs to $\Om\cap H^-$ or $R_a({a^+}\cap\Om) \cap H^-$ we keep the integrands unchanged. If a point belongs to $(\Om \cap a^-)\setminus H^-$ or to $R_a({a^+}\cap\Om) \setminus H^-$, we virtually transport it in any free point of $H^-$ or $R_aH^-$, the overall mass being preserved.

We get
\[2\int_{H^-} J(\theta^-)^2 dv\le  2 \left( \sum_{i=2}^{n}\frac{1}{\mu_i(\Om)}\right) \int_{{H^-} }     \frac{J(\theta^-)^2}{\sin(\theta^-)^2}dv+2 \sum_{i=1}^{n} \frac{1}{\mu_{i+1}(\Om)}  \int_{{H^-} }      \frac{(J'(\theta^-) v_i)^2}{\sin(\theta^-)^2}dv.\]
Taking into account the symmetry of $H^-$ and of $J$, we get
$$\int_{{H^-} }      \frac{(J'(\theta^-) v_i)^2}{\sin(\theta^-)^2}dv= \frac 1n \int_{{H^-} }       (J'(\theta^-)  )^2 dv= \frac 1n \int_{{B^{|\Om|/2}} }       (J'(\theta)  )^2 dv.$$
Since
$$\sum_{i=1}^{n}\frac{1}{\mu_i(\Om)}\leq \frac{n}{n-1}\sum_{i=2}^{n}\frac{1}{\mu_i(\Om)}$$
and
$$\mu_1(B^{|\Om|/2})= \frac{\int_{H^-} \left(J'(\theta^-)^2 +  \frac{J(\theta^-)^2}{\sin(\theta^-)^2}\right)dv}{\int_{H^-} J(\theta^-)^2 dv},$$
we conclude the proof.
 \end{proof}

\section{Further remarks and open questions}\label{sec_density}

\noindent{\bf Extension to densities.} The proofs of Theorems \ref{bmn05} and \ref{bmn06} and of their corollaries are exclusively based on mass transplantation. Once the topological result of Section \ref{bmnsec2} can be applied to identify the suitable family of test functions, the geometry of $\Om$ is not anymore relevant. In the spirit of \cite{BH19},   all results of the paper extend naturally to densities. A further work \cite{Ma22} on extremal densities for Neumann eigenvalues on the sphere is in preparation.

Shortly, assume that $\Om  \sq \Sn$ is an open Lipschitz set and that $\rho : \Om \rightarrow [0,1]$ is a measurable function such that $\mathrm{essinf}_\Om \rho >0$.

We consider the well posed eigenvalue problem in $H^1(\Om)$,
\[\mu_k(\rho)= \inf_{S\in{\mathcal S}_{k+1}} \sup_{u \in S\sm \{0\}} \frac{\int_{\Om} \rho|\nabla u|^2 }{\int_{\Om}\rho u^2 }.\]
Then, Theorem \ref{bmn06} reads
\[\mu_2(\rho)\leq \mu_1\left(B^{\frac{1}{2}\int_{\Sn}\rho}\right).\]
All the other estimates for $\mu_1(\rho)$ follow from this inequality. In particular, if $m < \frac{|\mathbb{S}^2|}{2}$,   $ \int_{\Sn}\rho=m$ and $\rho=0$ on $B^m$, then
\[\mu_1(\rho)\leq\mu_1(B^m).\]

\medskip

\noindent{\bf Going beyond $B^m$ for $\mu_1$.} Two questions are in order.
\medskip

  \noindent{\bf [Q1.]} {\it For $m < \frac{|\mathbb{S}^2|}{2}$, can one remove the constraint $\Om \sq \Sn \setminus B^m$?}
  The answer is not clear. One may think that it is only a technical difficulty in the construction of the test functions  in order to remove the inclusion condition $\Om \sq \Sn\sm B^m$, but the answer is more involved.  An observation which is worth to be noticed is  that numerical computations searching for the optimal domain on $ {\mathbb{S}^2}$, when they are performed in the exterior of $B^m$, are very stable and clearly lead to the spherical cap. This is coherent with the result of Corollary \ref{bmn08.1}. Once the inclusion constraint is removed, a certain numerical instability leading to some homogenization of the optimal density  can be observed, a phenomenon to which we do not have a precise explanation (for further details see \cite{Ma22}).   Yet, this does not imply that the result would be false on the whole sphere, but indicates that it is indeed false in the class of densities. This phenomenon is new.

As a consequence,   even  if true for every $\Om\sq \Sn$, $|\Om|<\frac{|\Sn|}{2}$, we do not expect the inequality $\mu_1(\Om)\leq \mu_1(B^m)$ could be proved by means of mass transplantations, like ours. Otherwise, this would necessarily generalize to $\mu_1(\rho)$, which is being contradicted  by the  numerical example given in Table \ref{fig:rho_pl_values} and  Figure \ref{fig:rho_pl}.  
In Table \ref{fig:rho_pl_values} we describe a piecewise affine, radial density, denoted $\rho^{pl}$. Its  graph is given in  Figure ~\ref{fig:rho_pl}. Then $\int_{\mathbb{S}^2} \rho^{pl} = 6< 2\pi$ and $\mu_1(\rho^{pl}) \approx 2.213185 > 2.071487 \approx \mu_1(B^6)$. The numerical computation of  $\mu_1(\rho^{pl})$ leads to a one dimensional eigenvalue problem and, although not certified theoretically, gives a strong indication about the validity of the result.

\begin{table}[b]
  \begin{tabular}{|l||l|l|l|l|l|}
    \hline
    $\theta$ & 0 & 1.3 & 1.4 & 3.14159265 \\
    \hline
    $\rho^{pl}(\theta)$ &  1 & 1 & 0.19480547 & 0.04829935   \\
    \hline
  \end{tabular}
  \smallskip
  \caption{}
  \label{fig:rho_pl_values}
\end{table}

 \begin{figure}
   \centering
   \includegraphics[width=0.40\textwidth]{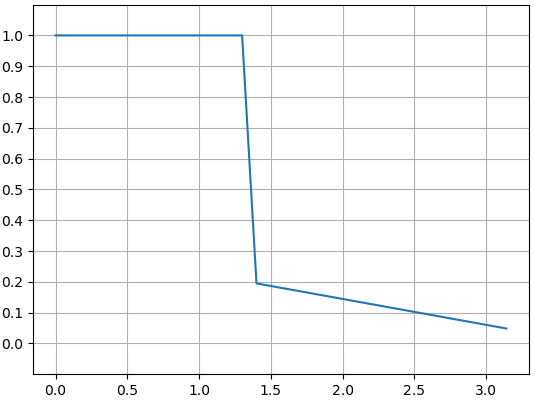}
   \includegraphics[width=0.32\textwidth]{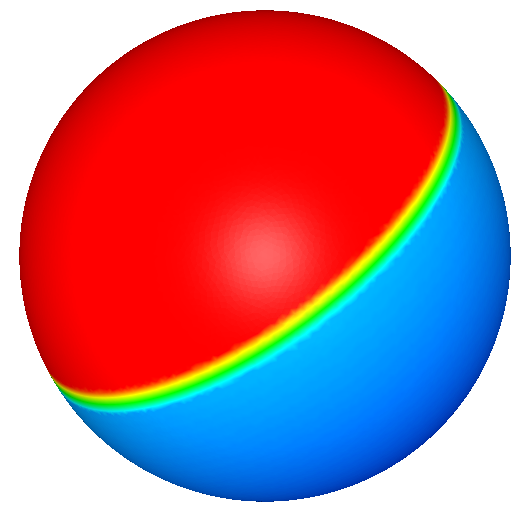}
   \caption{$\rho^{pl}$ plotted with respect to the polar angle $\theta \in [0,\pi]$ (left) and the same density on the sphere (right).}
   \label{fig:rho_pl}
 \end{figure}

 \EEE
 \medskip

  \noindent{\bf [Q2.]}  {\it What happens if $m \ge  \frac{|\mathbb{S}^2|}{2}$?} 
 Is there any chance that the spherical cap of prescribed volume continues to be maximal? The analysis performed in \cite{AB95} leads to the suggestion that this should not be the case. A strong argument is that the first Neumann eigenvalues on the hemisphere of ${\mathbb S}^2$ and the full sphere do coincide. Thus, the eigenvalue is not anymore decreasing when the measure of the geodesic ball increases, at least in a neighborhood of $4\pi$.  Numerical computations performed in \cite{Ma22} suggest the existence of sets having higher first eigenvalue than the spherical cap of the same measure.

\begin{remark}\rm {\bf (Euclidean version of Theorem  \ref{bmn05}).}
 The inequality proved in Theorem \ref{bmn05} comes as a consequence of the construction of the test functions and relies on the arguments of \cite{WaXi18} and \cite{BBC20}. Clearly, based on the result of \cite{BH19} on the maximality of $\mu_2$ in the Euclidean space, the following inequality can be proved, in the lines of Theorem \ref{bmn05}: let $\Om \sq \R^n$ be bounded, open and smooth, then
 $$ \sum_{i=2}^{n}\frac{1}{\mu_i(\Om)}\ge \frac{n-1}{\mu_1(B_{Euc} ^{|\Om|/2})}.$$
 This improves the result of \cite{BH19}.
 \end{remark}

\begin{remark}\rm
 The starting point of both results of \cite{WaXi18} and \cite{BBC20} is the conjecture by Ashbaugh and Benguria   that
$$\sum_{i=1}^{n}\frac{1}{\mu_i(\Om)}\ge \frac{n}{\mu_1(B_{Euc}^{|\Om|})}$$
in the Euclidean space $\R^n$ (and similarly on $\Sn$).
In view of Theorem \ref{bmn05},
we may naturally ask whether
$$\sum_{i=2}^{n+1}\frac{1}{\mu_i(\Om)}\ge \frac{n}{\mu_1(B ^{|\Om|/2})}$$
on $\Sn$  (and similarly in the Euclidean space $\R^n$).  This conjecture is supported  by numerical computations on ${\mathbb S}^2$ (see \cite{Ma22}).
Based on the multiplicity  of the second eigenvalue on the union of two equal spherical caps and on the conjecture above on $\mu_1$, one may
also suspect that
$$\sum_{i=2}^{2n+1}\frac{1}{\mu_i(\Om)}\ge \frac{2n}{\mu_1(B ^{|\Om|/2})}.$$
However, this last assertion fails to be supported by numerical computations.

\end{remark}

\section{Appendix - proof of  Lemma \ref{TopLemma}}

\begin{proof}
This is a consequence of the two following facts: first, by smoothing and a classical application of Sard's theorem we have the $\mathcal{C}^0$-density of vectors $V\in\X^*(M)$ with $\mathcal{C}^\infty$ regularity and nondegenerate zeros. Second for any two sufficently close, smooth and nondegenerate vector fields $V,W$, we have
\[\mathrm{Card}(V^{-1}(0))(V)=\mathrm{Card}(W^{-1}(0))\text{ mod }(4).\]
Indeed, letting $V_t=(1-t)V+tW$ we have an homotopy between $V$ and $W$ that never vanish on $\partial M$ when $\Vert V-W\Vert_{\mathcal{C}^0(M)}<\inf_{\partial M}|V|$. Then by a perturbation argument and parametric transversality theorem (see for instance \cite[Th. 7.1.1]{S11}) there is an homotopy $(H_t)_{0\leq t\leq 1}$ in $\X^*(M)$ such that $H_0$ (respectively  $H_1$) is an arbitrarily small perturbation of $V$ (respectively $W$) and so 
\[\mathrm{Card}(V^{-1}(0))=\mathrm{Card}(H_0^{-1}(0)),\ \mathrm{Card}(W^{-1}(0))=\mathrm{Card}(H_1^{-1}(0)),\]
and such that $H:[0,1]\times M\to TM$ is transversal to the null vector field. As a consequence $Z:=\{(t,x)\in [0,1]\times M:\ H_t(x)=0\}$ is a compact $1$-dimensional manifold that meets $\partial [0,1]\times M$ transversally but that does not meet $[0,1]\times \partial M$ (because the homotopy is in $\X^*(M)$). For every connected component $c$ of $Z$ there are four possibilities (due to the classification of $1$-dimensional manifold, see for instance \cite[Th 5.4.1]{S11}):
\begin{itemize}
\item $c$ is a circle ; in this case it has no influence on counting of zeroes of $H_0,\ H_1$.
\item $c$ is a curve (we write it $c(\tau)_{\tau \in [0,1]}$) such that $c(0)$ and $c(1)$ are both in $\{0\}\times M$. Consider then $\tilde{c}(\tau):=\left(\text{Id}_{[0,1]},S\right)\circ c(\tau)$ ; by the symmetry property of $V$ we know it is also a connected component of $Z$ with both ends in $\{0\}\times M$.

We claim that it is disjoint from $c$; indeed for every $\tau\in [0,1]$, $c(\tau)\neq \tilde{c}(\tau)$, however if $c$ and $\tilde{c}$ span the same curve then $\tilde{c}=c\circ \varphi$ for some continuous map $\varphi:[0,1]\to [0,1]$ with no fixed point, which is a contradiction. As a consequence $c(0),c(1),\tilde{c}(0),\tilde{c}(1)$ are all distincts, so the number of such zeroes is a multiple of $4$.
\item $c$ is a curve with both ends in $\{1\}\times M$; this is the same as above.
\item $c$ is a curve with $c(0)\in\{0\}\times M$, $c(1)\in \{1\}\times M$ ; this counts the same number of zeroes on both sides.
\end{itemize}
Thus the difference of the number of zeroes of $H_0$ and $H_1$ is a multiple of 4, which ends the proof.
\end{proof}

{\bf Acknowledgments.} The first author is thankful to Mark Ashbaugh for a deep insight of the question and fruitful suggestions.
 The authors were supported by ANR SHAPO (ANR-18-CE40-0013).

\bibliographystyle{plain}
\bibliography{biblio}

\begin{thebibliography}{10}

\bibitem{AsBe93}
Mark~S. Ashbaugh and Rafael~D. Benguria.
\newblock Universal bounds for the low eigenvalues of {N}eumann {L}aplacians in
  {$n$} dimensions.
\newblock {\em SIAM J. Math. Anal.}, 24(3):557--570, 1993.

\bibitem{AB95}
Mark~S. Ashbaugh and Rafael~D. Benguria.
\newblock Sharp upper bound to the first nonzero {N}eumann eigenvalue for
  bounded domains in spaces of constant curvature.
\newblock {\em J. London Math. Soc. (2)}, 52(2):402--416, 1995.

\bibitem{AsBe01}
Mark~S. Ashbaugh and Rafael~D. Benguria.
\newblock A sharp bound for the ratio of the first two {D}irichlet eigenvalues
  of a domain in a hemisphere of {$S^n$}.
\newblock {\em Trans. Amer. Math. Soc.}, 353(3):1055--1087, 2001.

\bibitem{BBC20}
Rafael~D. Benguria, Barbara Brandolini, and Francesco Chiacchio.
\newblock A sharp estimate for {N}eumann eigenvalues of the
  {L}aplace-{B}eltrami operator for domains in a hemisphere.
\newblock {\em Commun. Contemp. Math.}, 22(3):1950018, 9, 2020.

\bibitem{BH19}
Dorin Bucur and Antoine Henrot.
\newblock Maximization of the second non-trivial {N}eumann eigenvalue.
\newblock {\em Acta Math.}, 222(2):337--361, 2019.

\bibitem{Ch80}
Isaac Chavel.
\newblock Lowest-eigenvalue inequalities.
\newblock In {\em Geometry of the {L}aplace operator ({P}roc. {S}ympos. {P}ure
  {M}ath., {U}niv. {H}awaii, {H}onolulu, {H}awaii, 1979)}, Proc. Sympos. Pure
  Math., XXXVI, pages 79--89. Amer. Math. Soc., Providence, R.I., 1980.

\bibitem{FaWe18}
Mouhamed~Moustapha Fall and Tobias Weth.
\newblock Critical domains for the first nonzero {N}eumann eigenvalue in
  {R}iemannian manifolds.
\newblock {\em J. Geom. Anal.}, 29(4):3221--3247, 2019.

\bibitem{FrLa20}
Pedro Freitas and Richard~S. Laugesen.
\newblock Two balls maximize the third neumann eigenvalue in hyperbolic space,
  2020.

\bibitem{LL22}
J.J. Langford and R.S. Laugesen.
\newblock Maximizers beyond the hemisphere for the second neumann eigenvalue.
\newblock {\em Math. Ann.}, 2022.

\bibitem{Ma22}
Eloi Martinet.
\newblock {\em Shape optimization of Neumann eigenvalues on spheres}.
\newblock Work in preparation. 2022.

\bibitem{Ma03}
Ji\v{r}\'{\i} Matou\v{s}ek.
\newblock {\em Using the {B}orsuk-{U}lam theorem}.
\newblock Universitext. Springer-Verlag, Berlin, 2003.
\newblock Lectures on topological methods in combinatorics and geometry,
  Written in cooperation with Anders Bj\"{o}rner and G\"{u}nter M. Ziegler.

\bibitem{M97}
John~W. Milnor.
\newblock {\em Topology from the differentiable viewpoint}.
\newblock Princeton Landmarks in Mathematics. Princeton University Press,
  Princeton, NJ, 1997.
\newblock Based on notes by David W. Weaver, Revised reprint of the 1965
  original.

\bibitem{S11}
Anant~R. Shastri.
\newblock {\em Elements of differential topology}.
\newblock CRC Press, Boca Raton, FL, 2011.
\newblock With a foreword by F. Thomas Farrell.

\bibitem{Sz54}
G.~Szeg\"o.
\newblock Inequalities for certain eigenvalues of a membrane of given area.
\newblock {\em J. Rational Mech. Anal.}, 3:343--356, 1954.

\bibitem{WaXi18}
Qiaoling Wang and Changyu Xia.
\newblock On a conjecture of ashbaugh and benguria about lower eigenvalues of
  the neumann laplacian, 2018.

\bibitem{We56}
H.~F. Weinberger.
\newblock An isoperimetric inequality for the {$N$}-dimensional free membrane
  problem.
\newblock {\em J. Rational Mech. Anal.}, 5:633--636, 1956.

\end{thebibliography}

\end{document}